\documentclass[letterpaper,11pt,leqno]{article}
\usepackage{paper}

\hypersetup{pdftitle={Energy functionals on almost K\"ahler manifolds: I}}

\newcommand{\bib}{paper.bib}

\usepackage{mathrsfs}
\usepackage[shortlabels]{enumitem}
\newtheoremstyle{bfnoteonly}%
{}{}%
{\itshape}{}%
{\bfseries}{.}%
{ }%
{\thmnote{#3}}

\theoremstyle{bfnoteonly}

\numberwithin{equation}{section}

\usepackage{mathtools,math}

\let\oldforall\forall
\renewcommand{\forall}{\oldforall \, }

\let\oldexist\exists
\renewcommand{\exists}{\oldexist \: }

\DeclarePairedDelimiterX\Set[1]\{\}{%

#1
}

\usepackage{IEEEtrantools}
\DeclareMathSymbol{\intprod}{\mathbin}{MnSyC}{'270}
\newcommand{\rnum}[1]{\MakeUppercase{\romannumeral #1}}

\begin{document}

% Enter title:
\title{Energy functionals on almost \\ K\"ahler manifolds: I}

% Enter authors:

\author{Ken Wang, Zuyi Zhang, and Jiuru Zhou
%
% Enter affiliations and acknowledgements:
\thanks{ This work was partially supported by NSFC (China) Grants 12171417, 1197112.}}

% Enter date:
\predate{}
\postdate{\vspace{-1em}}
\date{}   

% Enter permanent URL (can be commented out):
%\available{https://github.com/pmichaillat/latex-paper}

% Enter affiliations and emails:
\newcommand{\Addresses}{{% additional braces for segregating \footnotesize
  \bigskip
  \footnotesize

  \textsc{School of Mathematical Sciences, Fudan University, Shanghai 100433, China}\par\nopagebreak
  \textit{E-mail address}: \url{kanwang22@m.fudan.edu.cn}

  \medskip

  \textsc{Beijing International Center for Mathematical Research, China} \par\nopagebreak
  \textit{E-mail address}: \url{zhangzuyi1993@hotmail.com}

  \medskip

   \textsc{School of Mathematical Sciences, Yangzhou University, Yangzhou 225002, Jiangsu, China}\par\nopagebreak
  \textit{E-mail address}: \url{zhoujr1982@hotmail.com}
}}

\maketitle

% Enter abstract:
\begin{abstract}
    In this paper, we consider the Donaldson gauge functional and the twisted Aubin functionals on almost K\"ahler manifolds. As in K\"ahler geometry, we generalize the inequality between Aubin functionals.
%   \smallskip

%   \noindent\textbf{Keywords.} Almost K\"ahler manifolds, Aubin functional, $\mathcal{D}_J^+$ operator, Donaldson gauge functional.
\end{abstract}

% Enter main text:
%%%%%%%%%%%%%%%%%%%%%%%%%%%%%%%%%%%%%%%%%%%%%%%%%%%%%%%%%%%%%%%%%%%%%%
\section{Introduction}
%%%%%%%%%%%%%%%%%%%%%%%%%%%%%%%%%%%%%%%%%%%%%%%%%%%%%%%%%%%%%%%%%%%%%%

In K\"ahler geometry, energy functionals---such as the Mabuchi functional \cite{Bando87} and the Aubin functionals \cite{Aubin84,Tian1987}---play a central role in the study of canonical metrics.  Bando and Mabuchi \cite{Bando87} proved that if a Fano manifold admits a K\"ahler--Einstein metric, then the Mabuchi functional is bounded from below; moreover, Tian's $\alpha$-invariant \cite{Tian1987} provides a key analytic criterion in this setting.  The Aubin functionals are also closely related to the $K$-energy and to geometric flows, including the $J$-flow introduced by Donaldson \cite{Donaldson99}; see, for instance, \cite{Chen00,Weinkove06,DingT92}.

From a variational viewpoint, these energies translate existence and uniqueness questions into coercivity, convexity, and properness properties along natural paths in the space of K\"ahler potentials.  Alongside Mabuchi's $K$-energy \cite{Mabuchi86}, the Ding functional \cite{Ding88} plays a parallel role in K\"ahler--Einstein geometry, and both interact with obstructions such as the Futaki invariant \cite{Futaki83}.  For constant scalar curvature and extremal K\"ahler metrics, the Calabi functional and its lower bounds \cite{Donaldson2005} further motivate an energy-theoretic approach, closely related to the moment-map picture on infinite-dimensional symmetric spaces \cite{Donaldson99,donaldson1999}.

In this paper we study generalizations of the Donaldson gauge functional and of the Aubin functionals in the almost K\"ahler setting.  Using the operator ${\mathcal{D}_J^+}$ introduced by Tan et al.\ \cite{TanWZZ22}, which extends the role of $\partial_J\bar\partial_J$, we define twisted Aubin-type functionals on almost K\"ahler manifolds and investigate their relationship with the Hermitian Donaldson gauge functional.  In particular, we establish an analogue of the classical inequality between Aubin functionals \cite{Aubin76}:
\begin{theorem} \label{thm:3.1}
    For the twisted Aubin functionals defined on an almost K\"ahler manifold $(M^{2m},J,g,\omega)$, the following relations hold:
    \[
    0\le\frac{1}{m+1}I'_\omega(\phi)\le J'_\omega(\phi)\le\frac{m}{m+1}I'_\omega(\phi)
    \]
    for any $\phi\in \mathcal{H}(\omega,J)$, where $\mathcal{H}(\omega,J)$ is the space of almost K\"ahler potentials. In particular, the equality, in any of the above three inequalities, holds if and only if $\phi$ is constant.
\end{theorem}

The paper is organized as follows. Section~2 introduces the operator $\mathcal{D}_J^+$ and recalls the primitive cohomologies of \cite{TsengY2012a,TsengY2012b}. Section~3 presents the Hermitian Donaldson gauge functional. Section~4 introduces the Hermitian Aubin functionals and proves \cref{thm:3.1}.

In a subsequent paper \cite{WangZ2025}, we will study the Mabuchi functional \cite{Mabuchi86}, the Ding functional \cite{Ding88}, and Tian's $\alpha$-invariant \cite{Tian1987} in the context of almost K\"ahler geometry.

%%%%%%%%%%%%%%%%%%%%%%%%%%%%%%%%%%%%%%%%%%%%%%%%%%%%%%%%%%%%%%%%%%%%%%
\section{Preliminary}\label{sec:preli}
%%%%%%%%%%%%%%%%%%%%%%%%%%%%%%%%%%%%%%%%%%%%%%%%%%%%%%%%%%%%%%%%%%%%%%

In this section we recall basic notions and introduce the operator ${\mathcal{D}_J^+}$ on higher-dimensional almost K\"ahler manifolds \cite{TanWWZ2025}.  We also review the Lefschetz and symplectic operators, following the framework of primitive cohomology in \cite{TsengY2012a,TsengY2012b}.  For standard material we refer to \cite{Aubin2012,Gauduchon2010}.

Let $(M,J)$ be an almost complex manifold of real dimension $2m$.  We write $C^\infty(M)$ for the space of smooth functions on $M$, let $\Lambda^{p,q}_J$ denote the bundle of $(p,q)$-forms with respect to $J$, and set $\Omega^{p,q}(M):=\Gamma(M,\Lambda^{p,q}_J)$.  The space of smooth $k$-forms is then $\Omega^{k}(M):=\bigoplus_{p+q=k}\Omega^{p,q}(M)$.

The almost complex structure $J$ acts on real $k$-forms by the involution
\begin{equation*}
    J:\alpha(\cdot,\ldots,\cdot) \mapsto (-1)^k\alpha(J\cdot,\ldots,J\cdot),\quad \alpha\in\Omega^k(M).
\end{equation*}
When $k=2$, this yields a splitting of $\Omega^2(M)$ into $J$-invariant and $J$-anti-invariant 2-forms,
\begin{equation*}
    \Omega^2 = \Omega_J^+ \oplus \Omega_J^-.
\end{equation*}
where $\Omega_J^+$ and $\Omega_J^-$ denote the spaces of $J$-invariant and $J$-anti-invariant 2-forms, respectively.  We then define the operators
\begin{equation*}
    \begin{aligned}
    d_J^+ &:= P_J^+d: \Omega^1 \longrightarrow \Omega_J^+, \\
    d_J^- &:= P_J^-d: \Omega^1 \longrightarrow \Omega_J^-,
    \end{aligned}
\end{equation*}
where $P_J^{\pm} = \frac{1}{2}(1\pm J)$ are algebraic projections on $\Omega^2(M)$.

Now suppose $(M^{2m},J,g,\omega)$ is an almost K\"ahler manifold; that is, $g$ is a $J$-compatible Riemannian metric and $\omega(\cdot,\cdot)=g(J\cdot,\cdot)$ is a closed $2$-form.  The adjoint operator $d^*:\Omega^k\rightarrow\Omega^{k-1}$ is defined by
\begin{equation*}
d^*=(-1)^{k+1}*_gd*_g,
\end{equation*}
where $*_g$ is the Hodge star operator defined using $g$. Consider the operator 
\begin{equation*}
d^-_{J}d^*:\Omega_J^-(M)\longrightarrow\Omega_J^-(M).
\end{equation*}
One checks that for all $\alpha,\beta\in\Omega_J^-(M)$,
\begin{equation*}
\langle d^-_Jd^*\alpha,\beta\rangle_g=\langle dd^*\alpha,\beta\rangle_g=\langle \alpha,dd^*\beta\rangle_g=\langle \alpha,d^-_Jd^*\beta\rangle_g,
\end{equation*}
hence $d^-_Jd^*$ is self-adjoint with respect to the $L^2$-inner product induced by $g$. Moreover, $d^-_Jd^*$ can be extended to a closed densely defined operator
\begin{equation*}
d^-_Jd^*:\Omega_J^-\otimes L_2^2(M)\longrightarrow \Omega_J^-\otimes L^2(M).
\end{equation*}

%\noindent Let $(M,\omega,J,g)$ be a closed almost Hermitian $2m$-manifold. Denote by
%$$
%C^\infty(M)_0:=\{f\in C^\infty(M)|\int_Mfd\mu_g=0\},
%$$
%see Tan-Wang-Wang \cite{TanWWZ2025}.

The following operator was first introduced by Lejmi \cite{Lejmi10E,Lejmi10S}.
\begin{proposition}[\cite{Lejmi10E}]
Let $(M,\omega,J,g)$ be a closed almost Hermitian 4-manifold. Define the following operator
\begin{align*}
    P: \Omega^-_J &\longrightarrow \Omega^-_J\\
    \psi&\longmapsto P^-_J(dd^*\psi).
\end{align*}
Then $P$ is a self-adjoint, strongly elliptic linear operator with a kernel consisting of $g$-self-dual-harmonic, $J$-anti-invariant 2-forms. Hence,
\begin{equation*}
    \Omega^-_J=\ker P \oplus d^-_J\Omega^1_\R.
\end{equation*}
\end{proposition}

Using this operator $P$, Tan et al. \cite{TanWZZ22} introduced the ${\mathcal{D}_J^+}$ operator as follows:
\begin{definition}
    Suppose $(M,\omega,J,g)$ is a closed almost K\"ahler 4-manifold. Define ${\mathcal{D}_J^+}$ as follows
    \begin{equation*}
        \begin{aligned}
        {\mathcal{D}_J^+}:C^\infty(M)_0 &\longrightarrow \Omega_J^+(M)\\
        {\mathcal{D}_J^+}(f) =dJdf+dd^* \sigma(f) &=dd^*(f\omega)+dd^*\sigma(f),
    \end{aligned}
    \end{equation*}
    
    where $\sigma(f)\in\Omega^-_J$ satisfies
    \begin{equation*}
        d^-_Jd^*(f\omega)+d^-_Jd^*\sigma(f)=0.
    \end{equation*}
    
\end{definition}
\begin{remark}
    In the original formulation in \cite{TanWZZ22}, the auxiliary operator
    \begin{equation*}
        \begin{aligned}
        {\mathcal{W}_J}:C^\infty(M)_0&\longrightarrow\Omega^1(M)\\
        f\phantom{(M)} &\longmapsto d^*(f\omega+\sigma(f)).
        \end{aligned}
    \end{equation*}
    is introduced. It satisfies
    \begin{equation*}
        d{\mathcal{W}_J}(f)={\mathcal{D}_J^+}(f), \quad d_J^-{\mathcal{W}_J}(f)=0.
    \end{equation*}
\end{remark}
One may further generalize the operator $\mathcal{D}_J^+$ to higher-dimensional almost K\"ahler manifolds; see \cite{TanWWZ2025}.  More precisely, for $f\in C^{\infty}(M)$, we define
\begin{equation*}
    \mathcal{D}_J^+(f) = dJdf + dd^*\sigma(f),
\end{equation*}
where $\sigma(f) \in \Omega_J^-(M)$ is chosen so that $d_J^-Jdf + d_J^-d^*\sigma(f) = 0$. Note that $\mathcal{D}_J^+(f)$ is $d$-exact. If $J$ is integrable, then $D_J^+(f) = 2\sqrt{-1}\partial_J \bar{\partial}_Jf$.

\begin{definition}[Tan et al. \cite{TanWWZ2025}]
    The space of almost K\"ahler potentials is defined by
    \begin{equation*}
    \mathcal{H}(\omega,J)=\{\phi\in C^{\infty}(M,\R)|\ \omega+\mathcal{D}_J^+\phi > 0 \}.
    \end{equation*}
\end{definition}
This space is contractible, infinite-dimensional, and affine.  It is an open subset of $C^{\infty}(M)$; hence for any $\phi \in \mathcal{H}(\omega,J)$ the tangent space $T_{\phi}\mathcal{H}(\omega,J)$ is naturally identified with $C^{\infty}(M)$.

\begin{definition}
    Suppose $\omega_1$ is another almost K\"ahler form on $M$. We say that $\omega$ and $\omega_1$ are \emph{related by an almost K\"ahler potential} if there is a smooth function $\phi\in  \mathcal{H}(\omega,J)$ such that 
    \begin{equation*}
    \omega_1=\omega+{\mathcal{D}_J^+}(\phi)>0.
    \end{equation*}
\end{definition}

In \cite{TanWWZ2025} it is shown that if $\phi_0$ solves the Laplace equation (in general, $\phi_0\neq\phi$), then
\begin{equation}\label{equ:lap}
    -\frac{1}{m}\Delta_{g}\phi_0=\frac{\omega^{m-1}\wedge(\omega_1-\omega)}{\omega^m}=\frac{\omega^{m-1}\wedge{\mathcal{D}_J^+}(\phi)}{\omega^m},
\end{equation}
then one can rewrite ${\mathcal{D}_J^+}(\phi)$ as
\begin{equation*}
{\mathcal{D}_J^+}(\phi)=dJd\phi_0+da(\phi_0),
\end{equation*}
where $a(\phi_0) = d^*\sigma(\phi_0) \in\Omega_{\R}^1(M)$.

The 1-form $a(\phi_0)$ satisfies the elliptic system.
\begin{equation}\label{equ:ellip}
    \begin{cases}
        d^{*}a(\phi_0) =0,\\
        d^-_Ja(\phi_0)=-d_J^-Jd\phi_0,\\
        \omega^{m-1}\wedge da(\phi_0)=0.
    \end{cases}
\end{equation}
The function $\phi_0$ is also called almost K\"ahler potential. If $J$ is integrable, then
\begin{equation*}
a(\phi_0)=0=\sigma(\phi),\ \ \phi=\phi_0.
\end{equation*}
For $m=2$, the system (\ref{equ:ellip}) reduces to the self-dual equation in Section 4 of \cite{WangZZZ2024}.

We now review primitive cohomology following Tseng–Yau \cite{TsengY2012a, TsengY2012b}. Define the \emph{Lefschetz operator} by 
\begin{equation*}
        \begin{aligned}
            L_\omega:\Omega^k(M)&\longrightarrow \Omega^{k+2}(M)\\
            \alpha\phantom{M}&\longmapsto \omega\wedge\alpha.
        \end{aligned}
\end{equation*}
A $k$-form $\alpha$ is called \emph{$k$-primitive} if 
    \begin{itemize}
        \item $L_\omega^{m-k}\alpha\neq0$,
        \item $L_\omega^{m-k+1}\alpha=0$.
    \end{itemize}
We denote the space of primitive $k$-form by $\mathcal{P}^k(M)$. In particular, every 1-form is 1-primitive.

\begin{proposition}[\cite{TsengY2012a}]
    For any $B_k\in\mathcal{P}^k(M)$, there exists $B_i\in\mathcal{P}^i(M)$ for $i=k\pm1$ such that
    \begin{equation}\label{equ:2.2}
        dB_k=B_{k+1}+L_\omega B_{k-1}.
    \end{equation}
\end{proposition}

This proposition gives rise to the \emph{symplectic operators} of $\omega$:
\begin{align*}
    \partial_{\pm}:\mathcal{P}^k(M) &\longrightarrow\mathcal{P}^{k\pm1}(M)\\
    B_k\phantom{M} &\longmapsto B_{k\pm1},
\end{align*}
where $B_{k\pm1}$ is specified by Equation (\ref{equ:2.2}).

The exterior derivative $d:\Omega^k(M)\rightarrow\Omega^{k+1}(M)$ admits the decomposition
\begin{equation*}
d=(A_J,\overline{\partial}_J,\partial_J,\overline{A}_J):\Omega^{p,q}(M)\rightarrow(\Omega^{p-1,q+2}\oplus\Omega^{p,q+1}\oplus\Omega^{p+1,q}\oplus\Omega^{p+2,q-1})(M).
\end{equation*}

\begin{lemma}[\cite{TsengY2012b}]\label{lem:1}
    If $\alpha=\alpha^{2,0}+\alpha^{0,2} \in \Omega_J^-$, then 
    \begin{equation*}
    d\alpha=\partial_{+}\alpha+\omega\wedge\partial_{-}\alpha.
    \end{equation*}
    where $\omega^{m-2}\wedge\partial_{+}\alpha=0$ and $\partial_{-}\alpha \in \Omega^1_{\R}(M)$.
\end{lemma}

\begin{example} \label{exmp:1} \leavevmode
    \begin{enumerate}[(i)]
        \item  From \cref{lem:1},
    \begin{equation*}
        dJd\phi_0 = \partial_+(Jd\phi_0) + \omega \wedge \partial_-(Jd\phi_0).
    \end{equation*}
    Using \cref{equ:ellip},
    \begin{equation*}
        \omega^{m-1} \wedge dJd\phi_0 = \omega^m \partial_-(Jd\phi_0) = -\frac{1}{m}\omega^m\Delta_{g}\phi_0.
    \end{equation*}
    Hence,
    \begin{equation} \label{equ:exmp1}
        \partial_-(Jd\phi_0) = -\frac{1}{m}\Delta_{g}\phi_0, \quad \partial_+(Jd\phi_0) = dJd\phi_0 + \frac{\omega}{m}\Delta_{g}\phi_0.
    \end{equation}

    \item 
    Similarly,
    \begin{equation*}
        0 = \omega^{m-1} \wedge da(\phi_0) = \omega^m \partial_-(a(\phi_0)).
    \end{equation*}
    Thus,
    \begin{equation}\label{equ:exmp2}
    \begin{cases}
        \partial_-(a(\phi_0)) = \partial_-(d^*\sigma(\phi_0)) = 0, \\
        P_J^- \partial_+ (a(\phi_0)) = d_J^-a(\phi_0) = -d_J^- Jd\phi_0.
    \end{cases}
    \end{equation}
    \end{enumerate}
\end{example}

The Hodge star operator $*_g$ and the symplectic Hodge star $*_{\omega}$ are related by the following.
\begin{lemma}[\cite{weil1971, TsengY2012a}]
    For any $\alpha\in\Omega^k(M)$,
    \begin{equation*}
    *_g\alpha=\mathcal{J}(*_{\omega}\alpha):=\sum_{(p,q)}(\sqrt{-1})^{p-q}\prod^{p,q}(*_{\omega}\alpha),
    \end{equation*}
    where $\prod\limits^{p,q}$ denotes projection to $(p,q)$-forms.
\end{lemma}

For the operator $\mathcal{J}:\Omega^k(M)\rightarrow\Omega^k(M)$ above, we compute a case needed later. If $\alpha\in\Omega^{1,0}(M)\oplus\Omega^{0,1}(M)$ and $\{dz_i\}$ is the local coframe of the $-\sqrt{-1}$ eigenspace of $J$, then
\begin{equation}\label{equ:Jact}
    \begin{aligned}
        J\alpha &=J\sum_{i}(\alpha^{i}dz_i+\alpha^{\bar i} d\bar{z_i})\\
    &=-\sum_{i}(\alpha^{i}(-\sqrt{-1})dz_i+\alpha^{\bar i}\sqrt{-1}d\bar{z_i})\\
    &=\sum_{i}(\alpha^{i}\sqrt{-1}dz_i+\alpha^{\bar i}(-\sqrt{-1})d\bar{z_i})\\
    &=\mathcal{J}\alpha.
    \end{aligned}
\end{equation}
A similar computation shows that $J\alpha=\mathcal{J}\alpha$ for $\alpha\in\Omega^{2,0}(M)\oplus\Omega^{0,2}(M)$.

\begin{lemma}[\cite{TsengY2012b}]\label{lem:lefid}
    For the symplectic operators $\partial_{\pm}$, the following identities hold:
    \begin{itemize}
        \item $\partial_{+}^2=\partial_{-}^2=0$,
        \item $dL_\omega^k=L_\omega^kd$,
        \item $*_{\omega}\frac{L_\omega^r}{r!}B_k=(-1)^{\frac{k(k+1)}{2}}\frac{L_\omega^{m-r-k}}{(m-r-k)!}B_k$,
        \item $*_g\frac{L_\omega^r}{r!}B_k=(-1)^{\frac{k(k+1)}{2}}\frac{L_\omega^{m-r-k}}{(m-r-k)!}\mathcal{J}B_k$,
    \end{itemize}
    for any $B_k\in\mathcal{P}^k(M)$ and integers $r,k$.
\end{lemma}
%%%%%%%%%%%%%%%%%%%%%%%%%%%%%%%%%%%%%%%%%%%%%%%%%%%%%%%%%%%%%%%%%%%%%%
\section{Donaldson gauge functionals} \label{sec:Donldsonfunctional}
%%%%%%%%%%%%%%%%%%%%%%%%%%%%%%%%%%%%%%%%%%%%%%%%%%%%%%%%%%%%%%%%%%%%%%

In this section we define the Hermitian Donaldson gauge functional and compute its derivative.  From now on, we set $V:=\int_M\omega^m$ and denote by $*$ the Riemannian Hodge star operator $*_g$.  We fix $\phi\in \mathcal{H}(\omega,J)$ and write $\omega_1=\omega+{\mathcal{D}_J^+}(\phi)$.

\begin{definition}\label{def:donfun}
    The \emph{Hermitian Donaldson gauge functional} is defined on $\mathcal{H}(\omega,J)$ as 
    \begin{equation*}
            \rnum{2}_\omega(\phi) := \frac{1}{(m+1)V}                            \Bigg(\int_M\phi\sum_{k=0}^m\omega_1^k\wedge\omega^{m-k}-           \int_M\phi dd^*\sigma(\phi)\wedge\omega^{m-1} \Bigg)
    \end{equation*}
    where $\omega_1=\omega+{\mathcal{D}_J^+}(\phi)$.
\end{definition}

\begin{remark}
    If $J$ is integrable, then $\sigma(\phi)=0$ and the functional reduces to
    \begin{equation*}
        \rnum{2}_\omega(\phi) = \frac{1}{(m+1)V} \int_M \phi ( \sum_{k=0}^m \omega_1^k \wedge \omega^{m-k} ). 
    \end{equation*}
    In the K\"ahler case, this functional $\rnum{2}_\omega$ was first introduced in Bando-Mabuchi \cite{Bando87}, and ---because of its relation to the moment-map picture in Donaldson’s work \cite{donaldson1999} --- it is often called the Donaldson gauge functional relative to $\omega$.
\end{remark}

\emph{Before analyzing its properties, we collect two technical lemmas.}

\begin{lemma}[\cite{TanWWZ2025}]\label{lem:3.3}
     Let $\partial_{\pm}$ be the symplectic operators associated to $\omega$. Then for all $\alpha\in\Omega_J^-$, one has
    \begin{equation*}
    d^{*}\alpha=(m-1)J\partial_{-}\alpha.
    \end{equation*}
\end{lemma}

\begin{proof}
    First, we may decompose $\alpha\in\Omega^-_J$ as $\alpha=\alpha^{2,0}+\alpha^{0,2}$, and we verify that
    \begin{equation*}
        \omega^{m-1} \wedge \alpha = 0,
    \end{equation*}
    which implies that $\alpha$ is primitive, i.e.\ $\alpha\in \mathcal{P}^2(M)$.
    
    Using repeatedly the identities in \cref{lem:lefid}, we compute
    \begin{align*}
        d^{*}\alpha=d^{*}L_\omega^0\alpha&={(-1)^3{*}d((-1)^{3}\frac{L_\omega^{m-2}}{(m-2)!}\mathcal{J}\alpha)}\\
        &=-{*}d(\frac{L_\omega^{m-2}}{(m-2)!}\alpha)\\
        &=-{*}\frac{L_\omega^{m-2}}{(m-2)!}d\alpha\\
        &=-{*}\left(\frac{L_\omega^{m-2}}{(m-2)!}\partial_{+}\alpha+\frac{L_\omega^{m-2}}{(m-2)!}\omega\wedge\partial_{-}\alpha \right)\\
        &=-{*}\frac{L_\omega^{m-1}}{(m-2)!}\partial_{-}\alpha\\
        &=(m-1)\frac{L_\omega^{m-m+1-1}}{(m-m+1-1)!}\mathcal{J}\partial_{-}\alpha\\
        &=(m-1)J\partial_{-}\alpha.
    \end{align*}
\end{proof}

Let $\alpha=\alpha^{2,0}+\alpha^{0,2} \in \Omega_J^{-}(M)$, which is a primitive form.  Applying the $(p,q)$-decomposition, we can write
\begin{equation*}
d \alpha  =\partial_{J} \alpha^{2,0}+\bar{\partial}_{J} \alpha^{0,2}
+\partial_{J} \alpha^{0,2}+\bar{\partial}_{J} \alpha^{2,0}+A_{J} \alpha^{2,0}+\bar{A}_{J} \alpha^{0,2}.
\end{equation*}
Denote $ \beta_{0}(\alpha) $ by the (1,2)+(2,1)-part of $\partial_+ \alpha$ in \cref{lem:1}. Then 
\begin{equation}
\begin{aligned}
d \alpha & =\partial_{J} \alpha^{2,0}+\bar{\partial}_{J} \alpha^{0,2}+\beta_{0}(\alpha)+\beta_{1}(\alpha) \\
& =\partial_{+} \alpha+\omega \wedge \partial_{-} \alpha,
\end{aligned}
\end{equation}
where $ \beta_{1}(\alpha)=\omega \wedge \partial_{-} \alpha $. In particular, there are no 3-primitive forms on closed 4-manifolds; hence $d\alpha = \omega \wedge \partial_- \alpha$.

For simplicity, we write $\beta_{0}(\phi)$ for $\beta_{0}(\sigma(\phi))$. Then we have the following lemma.
\begin{lemma}\label{lem:5}
    \begin{equation*}
    \begin{aligned}
        -(m-1)^2 \int_MJ \partial_- \sigma(f)\wedge \partial_- \sigma(\phi)\wedge\omega^{m-1-k}\wedge\omega_1^k = (A)_{f,\phi} + (B)_{f,\phi} = (A)_{\phi,f} + (B)_{\phi,f}
    \end{aligned}
    \end{equation*}
    where
    \begin{equation*}
    \begin{aligned}
        (A)_{u,v} &= \int_M u ( dd^*\sigma(v) \wedge \omega + (m-1)dJ\beta_0(v)) \wedge\omega^{m-2-k}\wedge\omega_1^k \\
        (B)_{u,v} &= (m-1)^2\int_M J \partial_- \sigma(u) \wedge \beta_0(v) \wedge\omega^{m-2-k}\wedge\omega_1^k.
    \end{aligned}
    \end{equation*}
    Moreover, we have
    \begin{equation*}
    \begin{aligned}
        -(m-1)^2 \int_MJ \partial_- \sigma(\phi)\wedge \partial_- \sigma(\phi)\wedge\omega^{m-1-k}\wedge\omega_1^k \geq 0
    \end{aligned}
    \end{equation*}
    for all integer $k\ge0$ such that $m-1-k\ge1$. In particular, if $J$ is integrable, then $\sigma(\phi)=0$ and so $\beta_0(\phi)=0$. Therefore, the above identities are trivial when $J$ is integrable.
\end{lemma}

\begin{proof}
We now prove the identity
\begin{equation} \label{eq:lem5_1}
    -(m-1)^2 \int_MJ \partial_- \sigma(f)\wedge \partial_- \sigma(\phi)\wedge\omega^{m-1-k}\wedge\omega_1^k = (A)_{f,\phi} + (B)_{f,\phi}
\end{equation}
The corresponding identity with $f$ and $\phi$ swapped follows by the same argument.  Applying \cref{lem:3.3}, we find
\begin{equation*}
\begin{aligned}
    (A)_{f,\phi}     =& {-}(m-1)\int_Mdf \wedge (J\partial_-\sigma(\phi)\wedge \omega + J\beta_0(\phi))\wedge\omega^{m-2-k}\wedge\omega_1^k
    \\
    =& {-}(m-1) \int_{M}\left\{d f \wedge (J\partial_-\sigma(\phi) \wedge \omega+J \beta_{0}(\phi))\right\}^{(2,2)} \wedge \omega^{m-2-k} \wedge \omega_{1}^k
    \\
    =&(m-1) \int_{M} Jd f \wedge (\partial_-\sigma(\phi) \wedge \omega+ \beta_{0}(\phi)) \wedge \omega^{m-2-k} \wedge \omega_{1}^k.
\end{aligned}
\end{equation*}
where in the second line we restrict to the $(2,2)$-part, since only $J$-invariant $(2,2)$-forms contribute.

On the other hand, since
\begin{equation*}
    \partial_{+} \alpha= \partial_{J} \alpha^{2,0}+\bar{\partial}_{J} \alpha^{0,2}+\beta_{0}(\alpha),
\end{equation*}
we have
\begin{equation*}
\begin{aligned}
    (A)_{f,\phi} =&(m-1)\int_M Jdf \wedge d\sigma(\phi) \wedge\omega^{m-2-k}\wedge\omega_1^k
    \\
    =&(m-1)\int_MdJdf\wedge\sigma(\phi)\wedge\omega^{m-2-k}\wedge\omega_1^k
    \\
    =&{-}(m-1)\int_Md_J^-d^*\sigma(f)\wedge\sigma(\phi)\wedge\omega^{m-2-k}\wedge\omega_1^k \quad \text{(since $d_J^-Jdf=-d^-_Jd^*\sigma(f)$)}
    \\
    =& {-}(m-1)\int_M d^*\sigma(f)\wedge d\sigma(\phi)\wedge\omega^{m-2-k}\wedge\omega_1^k
    \\
    =& {-}(m-1)^2\int_M J \partial_- \sigma(f) \wedge \beta_0(\phi) \wedge\omega^{m-2-k}\wedge\omega_1^k 
    \\
    &{}- (m-1)^2 \int_MJ \partial_- \sigma(f)\wedge \partial_- \sigma(\phi)\wedge\omega^{m-1-k}\wedge\omega_1^k.
\end{aligned}
\end{equation*}

To derive the second part, set $ f=\phi $. Then \cref{eq:lem5_1} becomes
\begin{equation*}
(A)_{\phi,\phi} + (B)_{\phi,\phi}  = {-}(m-1)^2 \int_MJ \partial_- \sigma(\phi)\wedge \partial_- \sigma(\phi)\wedge\omega^{m-1-k}\wedge\omega_1^k.
\end{equation*}
It remains to show the integral
\begin{equation*}
    (m-1)^2 \int_M \partial_- \sigma(\phi)\wedge J\partial_- \sigma(\phi)\wedge\omega^{m-1-k}\wedge\omega_1^k
\end{equation*}
is nonnegative. By Lemma~4.2.1 in \cite{Gauduchon2010}, at each $x \in M$ one may choose a $J$-adapted basis $\{e_1,Je_1,\ldots,e_m,Je_m\}$ of $T_xM$ with dual basis $\{e_1^*,Je_1^*,\ldots,e_m^*,Je_m^*\}$ so that $\omega(x)=\sum_je_j^*\wedge Je_j^*$ and $\omega_1=\sum_j\lambda e_j^*\wedge Je_j^*$ with $\lambda_j>0$. Moreover, there exists a positive $J$-(1,1) form $\omega^{(k)}$ and $\omega^{(k)}=\sum_j\mu_j^{(k)}e_j^*\wedge Je_j^*$ with $\mu_j^{(k)}>0$ and $(\omega^{(k)})^{m-1}=\omega^{m-1-k}\wedge\omega_1^k$. 

Writing $\partial_-\sigma(\phi)|_x = \sum_j(a_je^*_j+b_jJe^*_j)$ gives
\begin{eqnarray*}
 \lefteqn{\partial_{-}\sigma(\phi)\wedge J\partial_{-}\sigma(\phi)\wedge\omega^{m-1-k}\wedge\omega_1^k|_x} \\
 & & = \sum_j(a_je^*_j+b_jJe^*_j)\wedge\sum_j(a_jJe^*_j-b_je^*_j)\wedge(\omega^{(k)})^{m-1}|_x \\
 & & = \Big(\sum_j(a_j^2+b_j^2 )\wedge e^*_j\wedge Je^*_j \Big)\wedge(\omega^{(k)})^{m-1}|_x \geq 0.
 \end{eqnarray*}
Therefore, $(A)_{\phi,\phi} + (B)_{\phi,\phi} \geq 0$, which completes the proof.
\end{proof}

\begin{definition} \label{def:twistdonfun}
    The \emph{twisted Donaldson gauge functional} is defined on $\mathcal{H}(\omega,J)$ by
    \begin{equation*}
        \begin{aligned}
            \rnum{2}'_\omega(\phi):={} &\frac{1}{V}\int_M\phi\omega^m+\frac{1}{V}\int_M\phi dJd\phi\wedge (\sum_{k=0}^{m-1}\frac{m-k}{m+1}\omega_1^{k}\wedge\omega^{m-k-1} )
            \\
            &{} +\frac{(m-1)^2}{V}\int_M J \partial_- \sigma(\phi) \wedge \partial_- \sigma(\phi) \wedge (\sum_{k=0}^{m-2}\frac{m-k}{m+1}\omega_1^{k}\wedge\omega^{m-k-2} )
        \end{aligned}
    \end{equation*}
    where $\omega_1 = \omega + \mathcal{D}_J^+(\phi)$.
\end{definition}

\begin{remark}
    When $\dim M=4$, there are no primitive 3-forms, so $\beta_0(\phi)=0$. In this case,
    \begin{equation*}
            \rnum{2}'_\omega(\phi) ={} \frac{1}{V}\int_M\phi\omega^m + \frac{1}{3V}\int_M\phi dJd\phi\wedge (\omega_1 + 2\omega)-\frac{2}{3V}\int_M\phi dd^*\sigma(\phi).
    \end{equation*}
\end{remark}
    
We compute the derivative $\tau'$ of $\rnum{2}_\omega'$ at a potential $\phi \in \mathcal{H}(\omega,J)$.  More precisely,
\begin{align*}
    \tau'_{\phi}(f) ={}& d\rnum{2}_\omega'|_{\phi}(f) = \left. \frac{d}{dt} \right|_{t=0} \rnum{2}_\omega'(\phi+tf) 
    \\
    ={}& \frac{1}{V} \int_M f\omega^m +  \frac{1}{V} \int_M (f dJd\phi + \phi dJdf) \wedge (\sum_{k=0}^{m-1}\frac{m-k}{m+1}\omega_1^{k}\wedge\omega^{m-k-1} )
    \\
    &{} + \frac{1}{V} \int_M \phi dJd\phi \wedge (\sum_{k=0}^{m-1}\frac{(k+1)(m-k+1)}{m+1}\omega_1^{k}\wedge\omega^{m-k-2} \wedge \mathcal{D}_J^+(f) )
    \\
    &{} + \frac{2(m-1)^2}{V}  \int_M J \partial_- \sigma(f) \wedge \partial_- \sigma(\phi) \wedge (\sum_{k=0}^{m-2}\frac{m-k}{m+1}\omega_1^{k}\wedge\omega^{m-k-2} )
    \\
    &{} +\frac{(m-1)^2}{V}\int_M J \partial_- \sigma(\phi) \wedge \partial_- \sigma(\phi) \wedge (\sum_{k=0}^{m-3}\frac{(k+1)(m-k-1)}{m+1}\omega_1^{k}\wedge\omega^{m-k-3} \wedge \mathcal{D}_J^+(f))
\end{align*}

After simplification and using $\omega_1 - \omega = \mathcal{D}_J^+(\phi)$, we obtain
\begin{align*}
    \tau'_{\phi}(f)  ={}& \frac{1}{V} \int_M f\omega_1^m +  \frac{2}{V} \int_M f dJd\phi \wedge (\sum_{k=0}^{m-1}\frac{m-k}{m+1}\omega_1^{k}\wedge\omega^{m-k-1} )
    \\
    &{} + \frac{1}{V} \int_M \phi dJd\phi \wedge (\sum_{k=0}^{m-2}\frac{(k+1)(m-k-1)}{m+1}\omega_1^{k}\wedge\omega^{m-k-2} \wedge \mathcal{D}_J^+(f) )
    \\
    &{} + \frac{2(m-1)^2}{V}  \int_M J \partial_- \sigma(f) \wedge \partial_- \sigma(\phi) \wedge (\sum_{k=0}^{m-2}\frac{m-k}{m+1}\omega_1^{k}\wedge\omega^{m-k-2} )
    \\
    &{} +\frac{(m-1)^2}{V}\int_M J \partial_- \sigma(\phi) \wedge \partial_- \sigma(\phi) \wedge (\sum_{k=0}^{m-3}\frac{(k+1)(m-k-1)}{m+1}\omega_1^{k}\wedge\omega^{m-k-3} \wedge \mathcal{D}_J^+(f) )
    \\
    &{} -\frac{1}{V}\int_M f \mathcal{D}_J^+(\phi) \wedge (\sum_{k=0}^{m-1}\omega_1^{k}\wedge\omega^{m-k-1} ).
\end{align*}
By definition of the $\mathcal{D}_J^+$ operator, the above computation yields the following proposition.
\begin{proposition}[] \label{prop:4}
    The derivative $\tau'$ of $\rnum{2}_\omega'$ at $\phi$ can be written as
    \begin{equation*}
    \begin{aligned}
    \tau'_{\phi}(f)  ={}& \frac{1}{V} \int_M f\omega_1^m +  \frac{1}{V} \int_M f dJd\phi \wedge (\sum_{k=0}^{m-1}\frac{m-2k-1}{m+1}\omega_1^{k}\wedge\omega^{m-k-1} )
    \\
    &{} - \frac{1}{V} \int_M f dd^*\sigma(\phi) \wedge (\sum_{k=0}^{m-1}\omega_1^{k}\wedge\omega^{m-k-1} )
    \\
    &{} + \frac{1}{V} \int_M \phi dJd\phi \wedge (\sum_{k=0}^{m-2}\frac{(k+1)(m-k-1)}{m+1}\omega_1^{k}\wedge\omega^{m-k-2} \wedge \mathcal{D}_J^+(f) )
    \\
    &{} + \frac{2(m-1)^2}{V}  \int_M J \partial_- \sigma(f) \wedge \partial_- \sigma(\phi) \wedge (\sum_{k=0}^{m-2}\frac{m-k}{m+1}\omega_1^{k}\wedge\omega^{m-k-2} )
    \\
    &{} +\frac{(m-1)^2}{V}\int_M J \partial_- \sigma(\phi) \wedge \partial_- \sigma(\phi) \wedge (\sum_{k=0}^{m-3}\frac{(k+1)(m-k-1)}{m+1}\omega_1^{k}\wedge\omega^{m-k-3} \wedge \mathcal{D}_J^+(f) ).
    \end{aligned}
    \end{equation*}
    In particular, if $J$ is integrable, then $\sigma(\phi)= 0$ and hence
    \begin{equation*}
    \begin{aligned}
        \tau'_{\phi}(f) = & \frac{1}{V} \int_M f\omega_1^m.
    \end{aligned}
    \end{equation*}
\end{proposition}
%%%%%%%%%%%%%%%%%%%%%%%%%%%%%%%%%%%%%%%%%%%%%%%%%%%%%%%%%%%%%%%%%%%%%%
\section{Hermitian Aubin Functionals}\label{sec:Aubinfunctional}
%%%%%%%%%%%%%%%%%%%%%%%%%%%%%%%%%%%%%%%%%%%%%%%%%%%%%%%%%%%%%%%%%%%%%%

In this section we study Aubin functionals on almost K\"ahler manifolds and prove \cref{thm:3.1}.

Motivated by \cref{lem:5} and in analogy with \cite[Proposition~4.2.1]{Gauduchon2010}, we define the following twisted Aubin functionals.
\begin{definition} \label{defn:twistedAubin}
   The \emph{twisted Aubin functionals} are defined as 
    \begin{align*}
        I'_\omega(\phi):={}&{-\frac{1}{V}}\int_M\phi dJd\phi\wedge (\sum_{k=0}^{m-1}\omega^{m-k-1}\wedge\omega_1^k)
        \\
        &{} -\frac{(m-1)^2}{V}\int_M J \partial_- \sigma(\phi) \wedge \partial_- \sigma(\phi) \wedge (\sum_{k=0}^{m-2} \omega_1^{k}\wedge\omega^{m-k-2} ),
        \\
        J'_\omega(\phi):={}&{-\frac{1}{V}}\int_M\phi dJd\phi\wedge(\sum_{k=0}^{m-1}\frac{m-k}{m+1}\omega^{m-k-1}\wedge\omega_1^k)\\
        & -\frac{(m-1)^2}{V}\int_M J \partial_- \sigma(\phi) \wedge \partial_- \sigma(\phi) \wedge (\sum_{k=0}^{m-2}\frac{m-k}{m+1}\omega_1^{k}\wedge\omega^{m-k-2} )
    \end{align*}
    where $\phi\in \mathcal{H}(\omega,J)$ and $\omega_1=\omega+{\mathcal{D}_J^+}(\phi)$.
\end{definition}

\begin{remark}
    \begin{enumerate}[(i)]
        \item If $J$ is integrable, then $\sigma(\phi)=0$, and these twisted functionals reduce to the Aubin functionals in K\"ahler geometry \cite{Aubin84}.
        \item In real dimension 4, one obtains
        \begin{align*}
        I'_\omega(\phi) ={}&{-\frac{1}{V}}\int_M\phi dJd\phi\wedge (\omega+ \omega_1) +\frac{1}{V}\int_M\phi dd^*\sigma(\phi),
        \\
        J'_\omega(\phi) ={}&{-}\frac{1}{3V}\int_M\phi dJd\phi\wedge (\omega_1 + 2\omega) +\frac{2}{3V}\int_M\phi dd^*\sigma(\phi).
        \end{align*}
    \end{enumerate}
\end{remark}

By \cref{lem:5,defn:twistedAubin}, as in K\"ahler geometry, we can establish an inequality between $I'_\omega$ and $J'_\omega$ (cf.\ \cite[Proposition~4.2.1]{Gauduchon2010}).  To prove \cref{thm:3.1}, we need the following lemma.
\begin{lemma}\label{lem:2}
    For each integer $k=0,1,\ldots,m-1$,
    \begin{equation*}
        \int_M\phi dJd\phi\wedge\omega^{m-1-k}\wedge\omega_1^{k}=-\int_Md\phi\wedge Jd\phi\wedge\omega^{m-1-k}\wedge\omega_1^{k}\le0,
    \end{equation*}
\end{lemma}

\begin{proof}
    The identity follows directly from the Stokes formula. To show that the integral is negative, the argument is similar to that of Lemma \ref{lem:5}.
\end{proof}

\begin{proof}[Proof of \cref{thm:3.1}]
    A straightforward calculation gives
    \begin{equation*}
        \begin{aligned}
             \frac{m}{m+1}I'_\omega(\phi) ={}& J'_\omega(\phi)-\frac{1}{(m+1)V} \sum_{k=0}^{m-1} \int_M k\phi dJd\phi\wedge \omega^{m-k-1}\wedge\omega_1^k
        \\
        &{}- \frac{(m-1)^2}{(m+1)V} \sum_{k=0}^{m-2} k \int_M J\partial_-\sigma(\phi) \wedge \partial_-\sigma(\phi) \wedge \omega^{m-k-1}\wedge\omega_1^k
        \end{aligned}
    \end{equation*}
    and also
    \begin{equation*}
    \begin{aligned}
        (m+1)J'_\omega(\phi) ={}& I'_\omega(\phi)-\frac{1}{V} \sum_{k=0}^{m-2} \int_M (m-k-1) \phi dJd\phi\wedge \omega^{m-k-1}\wedge\omega_1^k
        \\
        &{}-\frac{(m-1)^2}{V} \sum_{k=0}^{m-2}\int_M (m-k-1) J\partial_-\sigma(\phi) \wedge \partial_-\sigma(\phi) \wedge \omega^{m-k-2}\wedge\omega_1^k.
    \end{aligned}
    \end{equation*}
    Using \cref{lem:5}, \cref{lem:2},together with \cref{defn:twistedAubin}, one immediately obtains the inequality.
\end{proof}

As in K\"ahler geometry (cf. \cite[Proposition~4.2.2]{Gauduchon2010}), we obtain the following relations among the twisted Donaldson gauge functional and the twisted Aubin functionals.
\begin{proposition}
    For any $\phi\in \mathcal{H}(\omega,J)$, the following identities hold:
    \begin{enumerate}[(i)]
        \item 
        \begin{equation*}
        J_\omega'(\phi)+\rnum{2}_\omega'(\phi) = \frac{1}{V}\int_M\phi\omega^m,
        \end{equation*}
        \item 
        \begin{equation*}
        \begin{aligned}
        I_\omega'(\phi)-J_\omega'(\phi)-\rnum{2}_\omega'(\phi) ={}& {-} \frac{1}{V}\int_M\phi\omega_1^m +
        \frac{1}{V}\int_M\phi dd^*\sigma(\phi)\wedge(\sum_{k=0}^{m-1} \omega^{m-k-1} \wedge \omega_1^k)
        \\
        &{}- \frac{(m-1)^2}{V}\int_MJ\partial_-\sigma(\phi) \wedge \partial_-\sigma(\phi)\wedge(\sum_{k=0}^{m-1}\omega^{m-k-1}\wedge\omega_1^k).
        \end{aligned}
    \end{equation*}
    \end{enumerate}
\end{proposition}

%%%%%%%%%%%%%%%%%%%%%%%%%%%%%%%%%%%%%%%%%%%%%%%%%%%%%%%%%%%%%%%%%%%%%%
\section*{Acknowledgement}
%%%%%%%%%%%%%%%%%%%%%%%%%%%%%%%%%%%%%%%%%%%%%%%%%%%%%%%%%%%%%%%%%%%%%%

We gratefully thank Qiang Tan for insightful communications. The first-named author would like to  thank his advisor Z. L\"{u} for encouragement and support.

\appendix

% Enter appendix text:

\bibliographystyle{paper}
\bibliography{\bib}
\Addresses

\end{document}